\documentclass[12pt]{amsart}
\usepackage{fullpage}
\usepackage{amsfonts,amscd}
\usepackage{amssymb}
\usepackage{url}
\usepackage{graphicx}
\usepackage[english]{babel}

\usepackage{pstricks}
\usepackage{pst-plot}
\usepackage{pst-node}
\SpecialCoor

\theoremstyle{plain}
\newtheorem{theorem}                {Theorem}      [section]
\newtheorem{proposition}  [theorem]  {Proposition}
\newtheorem{corollary}    [theorem]  {Corollary}

\theoremstyle{definition}

\newtheorem{remark}       [theorem]  {Remark}

\numberwithin{equation}{section}

\def \R{{\mathbb R}}
\def \s{{\mathbb S}}

\def \imto{\hookrightarrow}

\usepackage{color}

\DeclareMathOperator{\cst}{constant}
\DeclareMathOperator{\grad}{grad}
\DeclareMathOperator{\trace}{trace}
\DeclareMathOperator{\Div}{div}

\numberwithin{equation}{section}
\def \i{\mathrm i}

\begin{document}

\title[]{Biconservative  surfaces}

\author{S.~Montaldo}
\address{Universit\`a degli Studi di Cagliari\\
Dipartimento di Matematica e Informatica\\
Via Ospedale 72\\
09124 Cagliari, Italia}
\email{montaldo@unica.it}

\author{C.~Oniciuc}
\address{Faculty of Mathematics\\ ``Al.I. Cuza'' University of Iasi\\
Bd. Carol I no. 11 \\
700506 Iasi, ROMANIA}
\email{oniciucc@uaic.ro}

\author{A.~Ratto}
\address{Universit\`a degli Studi di Cagliari\\
Dipartimento di Matematica e Informatica\\
Viale Merello 92\\
09123 Cagliari, Italia}
\email{rattoa@unica.it}

\begin{abstract}
In this work we obtain some geometric properties of biconservative
surfaces into a Riemannian manifold. In particular, we shall study the
relationship between biconservative surfaces and  the holomorphicity of a generalized Hopf function.
Also, we give a complete classification of CMC biconservative surfaces in a $4$-dimensional space form.
\end{abstract}



\thanks{Work supported by: PRIN 2010/11 -- Variet\`a reali e complesse: geometria, topologia e analisi armonica N. 2010NNBZ78 003 -- Italy; GNSAGA -- INdAM, Italy;
Romanian National Authority for Scientific Research, CNCS -- UEFISCDI, project number PN-II-RU-TE-2011-3-0108}

\maketitle

\section{Introduction}\label{intro}
For an immersion $\varphi:M^2\imto N^3(c)$ of an oriented surface into a three-dimensional space form of constant sectional curvature $c$, the classical Hopf function is defined by
\begin{equation}\label{eq:hopf-classical}
\Phi(z,\bar{z})=\langle B(\partial_z,\partial_z),\eta \rangle=\langle A(\partial_z),\partial_z \rangle\,,
\end{equation}
where $\eta$ is the global unit normal vector field, $B$ is the second fundamental form, $A$ is the shape operator  and, with respect to isothermal coordinates $(x,y)$ on $M^2$,
$$
z=x+\i y\,,\quad
\partial_z=\frac{1}{2}\left(\frac{\partial}{\partial x}-\i \frac{\partial}{\partial y}\right)\,,\quad
\partial_{\bar{z}}=\frac{1}{2}\left(\frac{\partial}{\partial x}+\i \frac{\partial}{\partial y}\right)\,.
$$
The Hopf function defined in \eqref{eq:hopf-classical} is the key ingredient in the proof of the famous Hopf's Theorem: {\it a CMC immersed sphere in $N^3(c)$ is a round sphere}. Hopf's proof is based on the following facts: (i) $\Phi$ is holomorphic if and only if $M$ is CMC; (ii) $\Phi=0$ at umbilical points;
(iii) on a topological sphere any holomorphic quadratic differential vanishes everywhere.\\

As for immersions $\varphi:M^2\imto N^n(c)$, $n\geq 4$, the function $\Phi$ cannot be defined and a fair substitute is the function
\begin{equation}\label{eq:Q-classical}
Q(z,\bar{z})=\langle B(\partial_z,\partial_z),H\rangle=\langle A_H(\partial_z),\partial_z \rangle\,,
\end{equation}
where $H=(1/2)\, \trace B$ is the mean curvature vector field.  A classical result, see \cite{Hof73, Y}, states that if the surface $M$ has parallel mean curvature (PMC) then $Q$ is holomorphic. The function $Q$ in \eqref{eq:Q-classical} was recently used by Loubeau, Oniciuc (see \cite{LO}) and, in particular, enabled them to prove several rigidity results for biharmonic surfaces. We point out that the functions $\Phi, \, Q$ above are defined \emph{locally} on the surface $M$. Indeed, to give them a global meaning, one should consider their associated quadratic differentials (see \cite{Hopf}). However, for the sake of simplicity, we decided to work with functions because most of our calculations and results are of a local nature.\\

A natural question is to ask whether the converse of the Hoffman and Yau result holds, that is, if a surface with $Q$ holomorphic is PMC. This question
is interesting also for surfaces $M^2\imto N^3(c)$. In this case, it turns out (see Proposition~\ref{pro:Q-holomorphic-A-gradf} below) that
$Q$ is holomorphic if and only if the following condition is satisfied :
\begin{equation}\label{eq:Q-holomorphic}
A(\grad f)=0\,,
\end{equation}
where we have put $H=f\, \eta$. Since $\det A = K^M - c\,$, where $K^M$ is the Gaussian curvature of $M$, it follows that, if $K^M\neq c$, then $Q$ is holomorphic if and only if $M$ is CMC. On the other hand, a solution of \eqref{eq:Q-holomorphic} which is not CMC must have $K^M=c$, and there exist examples. For instance, in ${\mathbb R}^3$, we shall show that the cone parametrized by
$$
X(u,v)=\left(\frac{1}{k}\,(1-\alpha v) \cos(k u), \frac{1}{k}\,(1-\alpha v) \sin(k u), \beta v\right)\,,
$$
where $\alpha$, $k$ and $\beta$ are constants satisfying $\alpha^2+k^2 \beta^2=k^2$ and $\alpha\neq 0$, has $Q$ holomorphic and it is not CMC (see Section~\ref{sec:Q-holomorphic} below for details).\\

The main motivation of this paper is to study the relationship between the holomorphicity of $Q$ and the constancy of the mean curvature in a more general context. More specifically, we shall study immersed surfaces $M^2\imto N^n$, the starting point being the fact that the definition \eqref{eq:Q-classical} of the function $Q$ can be adopted also for a surface into any Riemannian manifold $(N^n,h)$. Moreover, $Q=0$ describes the pseudo-umbilical points.\\

In this more general framework, the family of immersed surfaces $M^2\imto N^n$ which turns out to be suitable for our purposes is that of \emph{biconservative surfaces}. Biconservative immersions are a rapidly developping topic (see Section \ref{sec:stress-energy-tensor} below for complete definitions and details). In simple words, they are defined as immersions for which the tangential component of the bitension field vanishes.
In particular (see Theorem \ref{teo:Q-hol-h-constant-biconservative} below), we shall prove that, if $M^2\imto N^n$ is biconservative, then
the holomorphicity of $Q$ is equivalent to the constancy of the mean curvature.

A key ingredient to obtain this result is the symmetric $(0,2)$-tensor, associated to any immersion $\varphi:M^2\to (N^n,h)$, defined by
\begin{equation}\label{eq:bistress-introduction}
S_2=-2 |H|^2 g+ 4 A_H\,,
\end{equation}
where $g=\varphi^{\ast}h=\langle,\,\rangle$. The tensor $S_2$ is called the {\it stress-bienergy} tensor and, as we will describe in Section~\ref{sec:stress-energy-tensor}, it has a variational meaning. In particular, an immersed surface $\varphi:M^2\to (N^n,h)$ is biconservative if $\Div S_2=0$.
More specifically, we will show that Theorem~\ref{teo:Q-hol-h-constant-biconservative} is a consequence of the fact that, for a biconservative surface, $4|H|^2=\trace S_2=\cst$ if and only if  $S_2(\partial_z,\partial_z)=4 Q(z,\bar{z})$ is holomorphic.\\

If $\varphi:M^2\to N^3(c)$ is an immersed surface into a three-dimensional space form, then
\begin{equation}\label{eq:biconservative-intro}
\frac{1}{4} \Div S_2= A(\grad f)+\,f\, \grad f \,\, ,
\end{equation}
which implies that CMC surfaces are automatically biconservative. In \cite{CMOP2013}, followed by \cite{Fu}, the authors provided the complete classification of biconservative surfaces into a $3$-dimensional space form which are not CMC.\\

In codimension greater than or equal to 2, PMC surfaces into space forms are automatically biconservative, while CMC surfaces are not. Therefore, it makes sense to study CMC biconservative surfaces. In this paper, we shall consider the simplest case of codimension 2 and prove the following theorem:\\

{\bf Theorem~\ref{teo:m2n4biconservative}}. {\it Let  $M^2\imto N^4(c)$ be a CMC biconservative surface into a space form of constant sectional curvature $c\neq 0$. Then $M^2$ is PMC.}\\

Thus, there could exist a CMC biconservative surface in $N^4(c)$, which is not PMC, only when $c=0$, and in this case we shall show that,
locally, the surface is given by
$$
X(u,v)=(\gamma(u),v+a)=(\gamma^1(u),\gamma^2(u),\gamma^3(u),v+a),\quad a\in\R \,\, ,
$$
where $ \gamma:I\to \R^3$ is a curve in $\R^3$ with constant curvature.

\section{Biharmonic maps, stress-energy tensors and biconservative immersions}\label{sec:stress-energy-tensor}

As described by Hilbert in~\cite{DH}, the {\it stress-energy}
tensor associated to a variational problem is a symmetric
$2$-covariant tensor $S$ which is conservative at critical points,
i.e. with $\Div S=0$.

In the context of harmonic maps $\varphi:(M,g)\to (N,h)$ between two Riemannian manifolds,
that is critical points of the {\em energy} functional
\begin{equation}\label{energia}
E(\varphi)=\frac{1}{2}\int_{M}\,|d\varphi|^2\,dv_g \,\, ,
\end{equation}
the stress-energy tensor was studied in detail by
Baird and Eells in~\cite{PBJE} (see also \cite{AS} and \cite{BR}). Indeed, the Euler-Lagrange
equation associated to the energy functional \eqref{energia} is equivalent to the vanishing of the tension
field $\tau(\varphi)=\trace\nabla d\varphi$ (see \cite{ES}), and the tensor
$$
S=\frac{1}{2}\vert d\varphi\vert^2 g - \varphi^{\ast}h
$$
satisfies $\Div S=-\langle\tau(\varphi),d\varphi\rangle$. Therefore, $\Div S=0$ when the map is harmonic.

\begin{remark}\label{remark:conservative}
We point out that, in the case of isometric immersions, the condition $\Div S=0$ is always satisfied,
 since $\tau(\varphi)$ is normal.
\end{remark}

A natural generalization of harmonic maps are the so-called {\it biharmonic maps}: these maps are the critical points of the bienergy functional (as suggested by Eells--Lemaire \cite{EL83})
\begin{equation}\label{bienergia}
    E_2(\varphi)=\frac{1}{2}\int_{M}\,|\tau (\varphi)|^2\,dv_g \,\, .
\end{equation}
In \cite{Jiang} G.~Jiang showed that the Euler-Lagrange equation associated to $E_2(\varphi)$ is given by the vanishing  of the bitension field
\begin{equation}\label{bitensionfield}
 \tau_2(\varphi) = - \Delta \tau(\varphi)- \trace R^N(d \varphi, \tau(\varphi)) d \varphi  \,\, ,
\end{equation}
where $\Delta$ is the rough Laplacian on sections of $\varphi^{-1} \, (TN)$ that, for a local orthonormal frame $\{e_i\}_{i=1}^m$ on $M$, is defined by
$$
    \Delta=-\sum_{i=1}^m\{\nabla^{\varphi}_{e_i}
    \nabla^{\varphi}_{e_i}-\nabla^{\varphi}_
    {\nabla^{M}_{e_i}e_i}\}\,\,.
$$
The curvature operator on $(N,h)$, which also appears in \eqref{bitensionfield}, can be computed by means of
$$
    R^N (X,Y)= \nabla_X \nabla_Y - \nabla_Y \nabla_X -\nabla_{[X,Y]} \,\, .
$$

The study of the stress-energy tensor for the
bienergy was initiated  in \cite{GYJ2} and afterwards developed in \cite{LMO}. Its expression is \begin{eqnarray}\label{eq:stress-bienergy-tensor}
S_2(X,Y)&=&\frac{1}{2}\vert\tau(\varphi)\vert^2\langle X,Y\rangle+
\langle d\varphi,\nabla\tau(\varphi)\rangle \langle X,Y\rangle \\
\nonumber && -\langle d\varphi(X), \nabla_Y\tau(\varphi)\rangle-\langle
d\varphi(Y), \nabla_X\tau(\varphi)\rangle,
\end{eqnarray}
and it satisfies the condition
\begin{equation}\label{eq:2-stress-condition}
\Div S_2=-\langle\tau_2(\varphi),d\varphi\rangle,
\end{equation}
thus conforming to the principle of a  stress-energy tensor for the
bienergy.\\

If  $\varphi:(M,g)\imto (N,h)$ is an isometric immersion, then \eqref{eq:2-stress-condition} becomes
$$
\Div S_2=-\tau_2(\varphi)^{\top}\,.
$$
This means that isometric immersions with $\Div S_2=0$ correspond to immersions with va\-nishing tangential part of the corresponding bitension field. \\

In particular, an isometric immersion $\varphi:(M,g)\imto (N,h)$ is called {\it biconservative} if $\Div S_2=0$. The study of biconservative immersions is quite rich from the analytical point of view, because one has to understand whether certain fourth order differential equations admit solutions which do not correspond to minimal immersions. Moreover, in suitable equivariant context (see \cite{MOR}) such a study is the starting point to obtain non-existence of proper biharmonic immersions.\\

In this paper, we restrict our attention to isometric immersions $\varphi:M^{2}\imto (N^{n},h)$  from a surface into an $n$-dimensional Riemannian manifold.
In this case $\tau(\varphi)=2\,H$, $H$ being the mean curvature vector field, and, if we denote by $A_H$ the shape operator in the direction of $H$ , we can easily check that the stress energy tensor $S_2$ given in \eqref{eq:stress-bienergy-tensor} becomes
$$
S_2=-2 |H|^2 g+ 4 A_H\,\,,
$$
which justifies \eqref{eq:bistress-introduction}.
Then, in the case of surfaces $M^{2}\imto N^{n}(c)$ into a space of constant sectional curvature $c$, a straightforward calculation gives
\begin{equation}\label{eq:biconservative-general}
\frac{1}{2} \Div S_2= 2\trace A_{\nabla^\perp_{(\cdot)}{H}}(\cdot)
+\grad {|H|}^2\,,
\end{equation}
which becomes, if $n=3$, condition \eqref{eq:biconservative-intro} in the introduction.

\section{Surfaces into $N^3(c)$ with $Q$ holomorphic}\label{sec:Q-holomorphic}
Let $M^2\imto N^3(c)$  be an oriented surfaces into a three-dimensional space form of constant sectional curvature $c$ and denote by $g=\langle , \rangle$ the induced metric on $M^2$.
 By assumption, $M^2$ is orientable and then it is a one-dimensional complex manifold. If we consider local isothermal coordinates
$(U; x,y)$, then  $g= \lambda^2 ( dx^2 + dy^2)$ for some positive function $\lambda$ on $U$ and $\{\partial_x,\partial_y\}$
is positively oriented. Let us denote, as usual,
$$
z=x+\i y\,,\quad
\partial_z=\frac{\partial}{\partial z}=\frac{1}{2}\left(\frac{\partial}{\partial x}-\i \frac{\partial}{\partial y}\right)\,,\quad
\partial_{\bar{z}}=\frac{\partial}{\partial \bar{z}}=\frac{1}{2}\left(\frac{\partial}{\partial x}+\i \frac{\partial}{\partial y}\right)\,.
$$
Also, let $\eta$ be the global unit normal vector field, $B$  the second fundamental form and $A$  the shape operator. Then the mean curvature vector field
can be written as $H=f\,\eta$, where $f$ is the mean curvature function.
As we mentioned in the introduction, for surfaces  $M^2\imto N^3(c)$ the function
$$
\Phi(z,\bar{z})=\langle B(\partial_z,\partial_z),\eta \rangle=\langle A(\partial_z),\partial_z \rangle
$$
is holomorphic if and only if the surface is CMC. In fact, a straightforward computation, taking into account Codazzi's equation, gives
\begin{equation}\label{eq-phi-olomorphic}
8\,\partial_{\bar{z}}\langle A(\partial_z),\partial_z \rangle=2\,\lambda^2(f_x-\i f_y)=4\, \lambda^2 \partial_{{z}}(f)\,.
\end{equation}
Differently, if we consider the function
$$
Q(z,\bar{z})=\langle B(\partial_z,\partial_z),H\rangle=\langle A_H(\partial_z),\partial_z \rangle\,,
$$
taking into account \eqref{eq-phi-olomorphic}, we obtain
\begin{eqnarray}\label{eq-Q-dzbar}
8\,\partial_{\bar{z}}\langle B(\partial_z,\partial_z),H\rangle&=& 8\,(\partial_{\bar{z}}f)\, \Phi+8\, f \, (\partial_{\bar{z}}\Phi)\nonumber\\
&=& (f_x+\i\, f_y) [\langle A(\partial_x),\partial_x \rangle-\langle A(\partial_y),\partial_y \rangle-2\, \i \langle A(\partial_x),\partial_y \rangle]+2\,\lambda^2\,f (f_x-\i f_y) \nonumber\\
&=& 2\, f_x\, \langle A(\partial_x),\partial_x \rangle+2\, f_y\, \langle A(\partial_y),\partial_x \rangle
-2\,\i (f_x\, \langle A(\partial_x),\partial_y \rangle+ f_y\, \langle A(\partial_y),\partial_y \rangle)\nonumber\\
&=& 2 \lambda^2   \langle A(\grad f),\partial_x \rangle- 2\, \lambda^2 \, \i  \langle A(\grad f),\partial_y \rangle\,,
\end{eqnarray}
where in the third equality we have used that $2\,\lambda^2\,f=\langle A(\partial_x),\partial_x \rangle+\langle A(\partial_y),\partial_y \rangle$.
From  \eqref{eq-Q-dzbar} we have the following proposition.
\begin{proposition}\label{pro:Q-holomorphic-A-gradf}
Let $M^2\imto N^3(c)$ be an oriented surface into a space form of constant sectional curvature $c$. If $f$ is constant, then $Q(z,\bar{z})$ is holomorphic. Conversely, if $Q(z,\bar{z})$ is holomorphic, then $A(\grad f)=0$.
\end{proposition}

As a consequence of \eqref{eq-Q-dzbar}, if we denote by $K^M$ the Gaussian curvature of the surface $M$, we obtain
\begin{proposition}\label{pro:Q-holomorphic}
Let $M^2\imto N^3(c)$ be an oriented surface into a space form of constant sectional curvature $c$.
\begin{itemize}
\item[(a)] If $\det(A)=K^M-c\neq 0$, then $f$ is constant if and only if $Q(z,\bar{z})$ is holomorphic;
\item[(b)] If $f$ is not constant and  $Q(z,\bar{z})$ is holomorphic, then $K^M=c$.
\end{itemize}
\end{proposition}
A natural question is whether there exists a surface satisfying condition (b) of Proposition~\ref{pro:Q-holomorphic}. In the next proposition
we will show that such surfaces do exist in $\R^3$ and we characterize them completely.

\begin{proposition}\label{pro:q-holomrphic-r3}
Let $M^2\imto \R^3$ be an oriented flat surface. Assume that $|\grad f|> 0$ and $f> 0$ on $M$.
If $Q(z,\bar{z})$ is holomorphic, then (locally) the surface can be parametrized by
$$
X(u,v)=\left(\frac{1}{k}\,(1-\alpha\, v) \cos(k u), \frac{1}{k}\,(1-\alpha\, v) \sin(k u), \beta \,v\right)\,,
$$
where $\alpha$, $k$ and $\beta$ are constants satisfying $\alpha^2+k^2 \beta^2=k^2$ and $\alpha\neq 0$.
\end{proposition}
\begin{proof}
 Since  $Q(z,\bar{z})$ is holomorphic, we deduce that $A(\grad f)=0$. Let $\{e_1,e_2\}$ be
a local orthonormal frame of principal directions such that
$$
Ae_1=0\,, \quad Ae_2=2 f e_2\,\,.
$$
From $A(\grad f)=2 f (e_2f)e_2=0$ we have that $e_2f=0$.
Next, from the Codazzi equation
$$
\nabla_{e_1} A(e_2)-\nabla_{e_2}A(e_1)=A[e_1,e_2] \,\, ,
$$
we find
$$
(e_1f)\,e_2+f\,\omega_2^1(e_1)\,e_1-f\,\omega_2^1(e_2)\, e_2=0\,\,,
$$
where $\omega_i^j$ are the connection 1-forms defined by $\nabla e_i=\omega_i^j\, e_j$.
Thus
$$
\begin{cases}
\omega_2^1(e_1)=0\\
\omega_2^1(e_2)=(e_1 f)/f=e_1 (\ln f)\neq 0\,\,.
\end{cases}
$$
From here, since
$$
B(e_1,e_1)=0\,,\quad B(e_1,e_2)=0\,,\quad B(e_2,e_2)=2 f\, \eta\,\,,
$$
we obtain
\begin{equation}\label{eq:connection-es-Q-holomorphic}
\begin{array}{ll}
\overline{\nabla}_{e_1}e_1=0\,,&\overline{\nabla}_{e_1}e_2=0\,,\\
\overline{\nabla}_{e_2}e_1=-e_1(\ln f) \,e_2\,,&\overline{\nabla}_{e_2}e_2=e_1(\ln f) \,e_1+2 f\,\eta\,,\\
\overline{\nabla}_{e_1}\eta=-Ae_1=0\,,& \overline{\nabla}_{e_2}\eta=-Ae_2=-2f\, e_2\,,
\end{array}
\end{equation}
where $\overline{\nabla}$ denotes the connection in $\R^3$. Let now $\gamma$ be an integral curve of $e_1$ parametrized by arc-length, that is
$\gamma(v)=\bar{a}+v\, \bar{b}$, with $\bar{a}$ and $\bar{b}$ constant vectors of $\R^3$ with $|\bar{b}|=1$. Let $p_0\in M$ and $\sigma(u)$ an integral curve of $e_2$ with $\sigma(0)=p_0$. Then the surface can be locally parametrized by
\begin{equation}\label{eq:X-first}
X(u,v)=\sigma(u)+v\, \bar{b}(u)\,,
\end{equation}
where $ \bar{b}(u)=e_1(\sigma(u))$.
Now, from \eqref{eq:connection-es-Q-holomorphic},
$$
\bar{b}'(u)=\frac{d \bar{b}}{du}=\overline{\nabla}_{e_2}e_1=-e_1(\ln f) \,e_2
$$
and, as $[e_1,e_2]=-e_1(\ln f)\, e_2$,
$$
e_2(e_1(\ln f))=e_1(e_2(\ln f))-[e_1,e_2]\, \ln f=0\,.
$$
We can then put $e_1(\ln f)|_{\sigma(u)}=\alpha=\cst\neq 0$ and
\begin{equation}\label{eq:b1-e2}
\bar{b}'(u)=-\alpha\, e_2.
\end{equation}
Next, by construction, $\sigma'(u)=e_2(\sigma(u))$ and this implies that
$$
\overline{\nabla}_{\sigma'}e_2=\overline{\nabla}_{e_2}e_2=\alpha\,e_1+2f\, \eta = k\, N\,,
$$
where $N$ is a unit normal vector field along $\sigma$ and $k=\sqrt{\alpha^2+4 f^2}=\cst\neq 0$ along $\sigma(u)$.
Moreover,
$$
\overline{\nabla}_{\sigma'}N=\frac{1}{k} \overline{\nabla}_{e_2}(\alpha\,e_1+2f\, \eta)=\frac{1}{k}(-\alpha^2\,e_2-4 f^2\, e_2)=-k\,e_2\,,
$$
which implies that $\sigma$ is the arc-length parametrization of a circle in $\R^3$ of radius $1/k$, that is, up to
a global isometry of $\R^3$,
$$
\sigma(u)=\left(\frac{1}{k}\, \cos(k\,u),\frac{1}{k}\, \sin(k\,u),0 \right)\,.
$$
From \eqref{eq:b1-e2} we then get
$$
\bar{b}(u)=\left(-\frac{\alpha}{k}\,\cos(k\,u)+\beta_1,-\frac{\alpha}{k}\, \sin(k\,u)+\beta_2, \beta \right)
$$
where, since $|\bar{b}(u)|=1$, $\beta_1=\beta_2=0$ and $\alpha^2+k^2\,\beta^2=k^2$. Replacing the expression of $\sigma(u)$ and
$\bar{b}(u)$ in \eqref{eq:X-first} we find the desired expression.
\end{proof}

\begin{remark}\label{re:q-olo-s3}
Proposition~\ref{pro:q-holomrphic-r3} can also be stated for surfaces into $N^3(c)$, with $c\neq 0$. For example, it can be proved that a non constant mean curvature surface  $M^2\imto\s^3\subset\R^4$, with $Q(z,\bar{z})$ holomorphic, is (locally) parametrized by
$$
\begin{aligned}
X(u,v)=&\frac{1}{k} \left(\cos(k u),  \sin(k u), \sqrt{k^2-1},0\right)\cos v\\
&+ \frac{1}{k} \left(-\alpha\,\cos(k u), -\alpha\, \sin(k u), \frac{\alpha}{\sqrt{k^2-1}},k\,\beta\right)\sin v
\end{aligned}
$$
where $\alpha$, $k$ and $\beta$ are constants satisfying $\beta^2(k^2-1)=k^2-\alpha^2-1$ and $k>\sqrt{1+\alpha^2}$, $\alpha\neq 0$.

\end{remark}

We end this section with the following general result about immersed surfaces into a space form.

\begin{proposition}\label{pro-q-olo-omeo-sphere}
Let $M^2\imto\ N^3(c)$ be an immersed orientable surface. Assume that $M^2$ is a topological sphere.
Then $M^2$ is CMC if and only if $Q(z,\bar{z})$ is holomorphic.
\end{proposition}
\begin{proof}
If $M^2$ is CMC from Proposition~\ref{pro:Q-holomorphic-A-gradf} $Q(z,\bar{z})$ is holomorphic.
Conversely, if $Q(z,\bar{z})$ is holomorphic and $M^2$ is a topological sphere, then $Q(z,\bar{z})$ must vanish.
This implies that $M^2$ is pseudo-umbilical. Let $H=f\,\eta$, where $\eta$ is the unit normal vector field. If $f=0$, then $M$ is minimal, thus CMC.  We can then assume that there exists $p_0\in M$ such that $|f(p_0)|>0$ and put $A=\{p\in M\colon |f(p)|=|f(p_0)|\}$. The subset $A$ is non-empty and closed in $M^2$. We shall prove that $A$ is also open. Indeed, let $p_1\in A$. As $f(p_1)\neq 0$ there exists an open neighborhood $U$ of $p_1$ such that $f(p)\neq 0$ for all $p\in U$. Since $M^2$ is pseudo-umbilical and $f(p)\neq 0$ for any $p\in U$ we conclude that $U$ is umbilical in $N^3(c)$ and this implies that $f|_{U}=\cst$, that is $U\subset A$.
\end{proof}

\begin{remark}\label{ultimo-cesare} It is not difficult to prove that, in general, if $M^2\imto\ (N^n,\,h)$ is pseudo-umbilical and biconservative, then it is CMC (see \cite{BMO13}, Proposition 2.5).
\end{remark}

\section{Divergence free symmetric tensors}

We begin with two rather general properties that could be interesting by themselves.
\begin{proposition}\label{pro:T-ho-codazzi}
Let $T$ be a symmetric $(0,2)$-tensor field on a Riemannian surface $(M^2,g)$ and set $t=\trace T$. Assume that $M^2$ is orientable and $\Div T=0$. Then
\begin{itemize}
\item[(a)] $T(\partial_z,\partial_z)$ is holomorphic if and only if $t=\cst$;
\item[(b)] $T$ is a Codazzi tensor if and only if $t=\cst$.
\end{itemize}
\end{proposition}

\begin{proof} (a) -
We consider local isothermal coordinates and use the same notations as in Section~\ref{sec:Q-holomorphic}.
Then
$$
\begin{aligned}
T(\partial_z,\partial_z)=&\frac{1}{4} T(\partial_x-\i\partial_y,\partial_x-\i\partial_y)\\
=& \frac{1}{4} \{T(\partial_x,\partial_x)-T(\partial_y,\partial_y)-2\i\,T(\partial_x,\partial_y)\}\,.
\end{aligned}
$$
Thus
$$
8\,\partial_{\bar{z}} T(\partial_z,\partial_z)=A+\i B \,\, ,
$$
where we have set
$$
A=\partial_x T(\partial_x,\partial_x)-\partial_x T(\partial_y,\partial_y)+2 \partial_y T(\partial_x,\partial_y)
$$
and
$$
B=\partial_y T(\partial_x,\partial_x)-\partial_y T(\partial_y,\partial_y)-2 \partial_x T(\partial_x,\partial_y)\,\,.
$$

We now show that $A=B=0$ if and only if $t=\cst$.

First, using the standard formula for the covariant derivative of a $(0,2)$-tensor, the term $A$ can be rewritten as
$$
\begin{aligned}
A=&\left(\nabla_{\partial_x} T \right)(\partial_x,\partial_x)+2 T(\nabla_{\partial_x}\partial_{x}, \partial_x)-
\left(\nabla_{\partial_x} T \right)(\partial_y,\partial_y)-2 T(\nabla_{\partial_x}\partial_{y}, \partial_y)\\
&+ 2 \left(\nabla_{\partial_y} T \right)(\partial_x,\partial_y)+2 T(\nabla_{\partial_y}\partial_{x}, \partial_y)+
2 T(\partial_x, \nabla_{\partial_y}\partial_y )\\
=& \left(\nabla_{\partial_x} T \right)(\partial_x,\partial_x)-
\left(\nabla_{\partial_x} T \right)(\partial_y,\partial_y)+ 2 \left(\nabla_{\partial_y} T \right)(\partial_y,\partial_x)+
2 T(\nabla_{\partial_x}\partial_{x}, \partial_x)+
2 T(\partial_x, \nabla_{\partial_y}\partial_y )\,.
\end{aligned}
$$
Next, from $\Div T=0$, we obtain
 $$
 0=\Div T(\partial_x)= \left(\nabla_{\partial_x} T \right)(\partial_x,\partial_x)+\left(\nabla_{\partial_y} T \right)(\partial_y,\partial_x)\,,
 $$
 which implies that
 \begin{equation}\label{eq:divT0-1}
\left(\nabla_{\partial_y} T \right)(\partial_y,\partial_x)=-
\left(\nabla_{\partial_x} T \right)(\partial_x,\partial_x)\,.
 \end{equation}
 Using \eqref{eq:divT0-1}, the expression of $A$ becomes
 $$
 \begin{aligned}
 A=&-\left(\nabla_{\partial_x} T \right)(\partial_x,\partial_x)-
\left(\nabla_{\partial_x} T \right)(\partial_y,\partial_y)+2 T(\nabla_{\partial_x}\partial_{x}, \partial_x)
+2 T(\partial_x, \nabla_{\partial_y}\partial_y )\\
=&-\partial_x T(\partial_x,\partial_x)+2 T(\partial_x, \nabla_{\partial_x}\partial_x )
-\partial_x T(\partial_y,\partial_y)+2 T(\partial_y, \nabla_{\partial_x}\partial_y )\\
& +2 T(\nabla_{\partial_x}\partial_{x}, \partial_x)
+2 T(\partial_x, \nabla_{\partial_y}\partial_y )\\
=& -\partial_x(\lambda^2\, t)+4 T(\nabla_{\partial_x}\partial_{x}, \partial_x)+2 T(\nabla_{\partial_x}\partial_{y}, \partial_y)+2 T(\nabla_{\partial_y}\partial_{y}, \partial_x)\,.
\end{aligned}
 $$
Since $(x,y)$ are isothermal coordinates, computing the Christoffel symbols we have
$$
\nabla_{\partial_x}\partial_{y}=\nabla_{\partial_y}\partial_{x}=\frac{1}{\lambda}\left( \lambda_y\, \partial_x+\lambda_x\, \partial_y \right)\,,
$$
$$
\nabla_{\partial_x}\partial_{x}=\frac{1}{\lambda}\left( \lambda_x\, \partial_x-\lambda_y\, \partial_y \right)\,,\quad
\nabla_{\partial_y}\partial_{y}=\frac{1}{\lambda}\left( -\lambda_x\, \partial_x+\lambda_y\, \partial_y \right)\,,
$$
from which
$$
A=-\lambda^2\, \partial_x t\,.
$$
In a similar way, we obtain
$$
B=\lambda^2\, \partial_y t\,.
$$

(b) - Let $\{X_1,X_2\}$ be a local orthonormal frame on $M^2$. The tensor $T$ is a Codazzi tensor if
$$
\left(\nabla_{X_1} T\right)(X_2,X_i)=\left(\nabla_{X_2} T\right)(X_1,X_i)\,,\quad i=1,2\,.
$$
By definition of covariant derivative of a $(0,2)$-tensor, we have
\begin{eqnarray}\label{eq-tcodazzi1}
\left(\nabla_{X_1} T\right)(X_2,X_1)-\left(\nabla_{X_2} T\right)(X_1,X_1)&=&X_1T(X_1,X_2)
-T(\nabla_{X_1}X_2,X_1)-T(\nabla_{X_1}X_1,X_2)\nonumber\\
&&
-X_2T(X_1,X_1)+ 2T(\nabla_{X_2}X_1,X_1)\nonumber\\
&=&X_1T(X_1,X_2)+X_2T(X_2,X_2)-T(\nabla_{X_1}X_1,X_2)\nonumber\\
&&
-T(\nabla_{X_1}X_2,X_1)+2 T(\nabla_{X_2}X_1,X_1)-X_2t\,.
\end{eqnarray}
Next, from $\Div T=0$, we obtain
$$
X_1T(X_1,X_2)+X_2T(X_2,X_2)=T(\nabla_{X_1}X_1,X_2)+T(\nabla_{X_1}X_2,X_1)+2 T(\nabla_{X_2}X_2,X_2)\,,
$$
that substituted in \eqref{eq-tcodazzi1} gives
\begin{eqnarray}\label{eq-tcodazzi2}
\left(\nabla_{X_1} T\right)(X_2,X_1)-\left(\nabla_{X_2} T\right)(X_1,X_1)&=&2T(\nabla_{X_2}X_2,X_2)+2T(\nabla_{X_2}X_1,X_1)-X_2t\nonumber\\
&=& -X_2t\,,
\end{eqnarray}
where, for the last equality, we have used the fact that
$$
T(\nabla_{X_2}X_2,X_2)+T(\nabla_{X_2}X_1,X_1)=\omega_2^1(X_2)\, T(X_1,X_2)-\omega_2^1(X_2)\, T(X_2,X_1)=0\,.
$$
In a similar way, one can prove that
$$
\left(\nabla_{X_1} T\right)(X_2,X_2)-\left(\nabla_{X_2} T\right)(X_1,X_2)=X_1t\,.
$$
 \end{proof}

Let now $M^2\imto N^n$ be a surface into an $n$-dimensional Riemannian manifold and denote by $g$ the induced metric. Then the stress energy tensor $S_2$ is a symmetric $(0,2)$-tensor
on $M^2$ which is defined, according to \eqref{eq:bistress-introduction}, by
$$
S_2=-2 |H|^2 g+ 4 A_H\,.
$$
Taking the trace of $S_2$ we have
$$
4 |H|^2=\trace S_2\,.
$$
Moreover, since $g(\partial_z,\partial_z)=0$,
we have that
$$
S_2(\partial_z,\partial_z)=4\,\langle A_H(\partial_z),\partial_z\rangle=4\, Q(z,\bar{z})\,.
$$

Thus, as a direct consequence of Proposition~\ref{pro:T-ho-codazzi}, we obtain the following theorem.

\begin{theorem}\label{teo:Q-hol-h-constant-biconservative}
Let $M^2\imto (N^n,h)$ be a biconservative surface into an $n$-dimensional Riemannian manifold. Then
\begin{itemize}
\item[(a)] $Q(z,\bar{z})$ is holomorphic if and only if $|H|=\cst$;
\item[(b)] $S_2$ is a Codazzi tensor if and only if $|H|=\cst$.
\end{itemize}
\end{theorem}

\begin{corollary} \label{cor:cmc-bicons-top-sphere}
Let $M^2\imto\ (N^n,\,h)$ be a CMC biconservative surface. If $M^2$ is a topological sphere, then $M^2$ is pseudo-umbilical.
\end{corollary}

\begin{remark} The previous corollary is a converse of Remark~\ref{ultimo-cesare}.
\end{remark}

\section{CMC biconservative surfaces in $N^4(c)$}

In this section we shall study the case of surfaces $M^2\imto N^4(c)$ into a $4$-dimensional
space form of constant sectional curvature $c$. In this setting, using \eqref{eq:biconservative-general}, a surface is biconservative if
\begin{equation}\label{eq:biconservative-general-n4}
 2\trace A_{\nabla^\perp_{(\cdot)}{H}}(\cdot)+\grad {|H|}^2=0\,.
\end{equation}

Clearly, a surface with parallel mean curvature vector field (PMC), that is such that $\nabla^\perp{H}=0$, is biconservative and such surfaces are classified. In fact, D.~ Hoffman,  in his doctoral thesis [Stanford University, 1971], classified PMC surfaces in $\R^4$. Later, his result was extended to any space form of any dimension in \cite{chen73,Y}, yielding the following classification of PMC surfaces into $N^4(c)$:
\begin{itemize}
\item[(i)]  minimal surfaces into $N^4(c)$;
\item[(ii)] CMC surfaces (including minimal) into a hypersphere of $N^4(c)$;
\end{itemize}

A natural question is when a CMC biconservative surface $M^2\imto N^4(c)$ is PMC. As
any pseudo-umbilical surface in a four dimensional space form is CMC if
and only if it is PMC (see \cite{chen71}), from Corolarry~\ref{cor:cmc-bicons-top-sphere} it follows that a CMC
biconservative surface in $N^4(c)$ which is topologically a sphere is
PMC. When $M$ is not necessary a sphere, we provide the following answer

\begin{theorem}\label{teo:m2n4biconservative}
Let  $M^2\imto N^4(c)$ be a CMC biconservative surface into a space form of constant sectional curvature $c\neq 0$. Then $M^2$ is PMC.
\end{theorem}
 \begin{proof}
 If $M^2\imto N^4(c)$ is pseudo-umbilical, then, as we have already said, $M^2$ is PMC (see \cite{chen71}). Thus we can assume that the surface is not pseudo-umbilical and that the pseudo-umbilical points are isolated. The last assertion is a consequence of the fact that, from Theorem~\ref{teo:Q-hol-h-constant-biconservative},
 $Q(z,\bar{z})$ is holomorphic and $Q(z,\bar{z})=0$ precisely at pseudo-umbilical points.
 Now, assume that $\nabla^{\perp}H\neq 0$. Therefore, there exists an open subset $U\subset M$ such that $\nabla^{\perp}H(p)\neq 0$ for all $p\in U$,
and
 $k_1(p)\neq k_2(p)$, where we have denoted by $k_1$ and $k_2$ the eigenvalues of the endomorphism
 $A_H$.
Let $\{E_1,E_2\}$ be a local orthonormal frame field (defined on $U$) such that
 $$
 A_{\frac{H}{|H|}}(E_1)=\lambda_1\, E_1\,,\quad  A_{\frac{H}{|H|}}(E_2)=\lambda_2\, E_2\,,
 $$
 where $\lambda_i=k_i/|H|$, $i=1,2$.
Moreover, put $E_3=H/|H|$  and define $E_4$ such that it is a unit normal vector field orthogonal to $E_3$.
Then the frame
\begin{equation}\label{eq:frameei}
\{E_1,E_2,E_3,E_4\}
\end{equation}
can be extended to a local frame of $N^4(c)$ defined on an open subset $V$ of $N^4(c)$.  Let us denote, for convenience,  $A_i=A_{E_i}$, $i=3,4$, and note that
$$
\trace A_4=\langle A_4(E_1), E_1 \rangle+\langle A_4(E_2), E_2 \rangle=\langle B(E_1,E_1),E_4\rangle+\langle B(E_2,E_2),E_4\rangle=2\langle H, E_4 \rangle=0\,.
$$
Also, let $\omega_A^B$ be the connection $1$-forms on $V$ corresponding to $\{E_1,E_2,E_3,E_4\}$, that is
$$
\nabla^N E_A=\omega_A^B\, E_B\,.
$$
Then
$$
\nabla^{\perp}_{E_1} E_3=\omega_3^4(E_1)\, E_4\,,\quad \nabla^{\perp}_{E_2} E_3=\omega_3^4(E_2)\, E_4\,.
$$
As $M$ is CMC, then the biconservative condition \eqref{eq:biconservative-general-n4} becomes
\begin{eqnarray}\label{eq:biconservative-ei}
 \trace A_{\nabla^\perp_{(\cdot)}{E_3}}(\cdot) &=& A_{\nabla^\perp_{E_1}{E_3}}(E_1) +A_{\nabla^\perp_{E_2}{E_3}}(E_2) \nonumber\\
 &=& \omega_3^4(E_1)\, A_4 E_1+\omega_3^4(E_2)\, A_4 E_2=0\,,
\end{eqnarray}
which can be written as the system
\begin{equation}\label{eq:biconservative-system}
\begin{cases}
\omega_3^4(E_1)\langle A_4 E_1, E_1\rangle+\omega_3^4(E_2)\langle A_4 E_2, E_1\rangle=0\\
\omega_3^4(E_1)\langle A_4 E_1, E_2\rangle+\omega_3^4(E_2)\langle A_4 E_2, E_2\rangle=0\,.
\end{cases}
\end{equation}
If we regard system \eqref{eq:biconservative-system} as a linear system in the variables $\omega_3^4(E_1)$ and $\omega_3^4(E_2)$, then, since $(\omega_3^4(E_1))^2+(\omega_3^4(E_2))^2>0$ on $U$ (it represents the square of the norm of $\nabla^{\perp} E_3$), we must have
$$
0=\langle A_4 E_1, E_1\rangle \langle A_4 E_2, E_2\rangle-(\langle A_4 E_1, E_2\rangle)^2=
-(\langle A_4 E_1, E_1\rangle)^2-(\langle A_4 E_1, E_2\rangle)^2=-|A_4 E_1|^2\,,
$$
where in the second equality we have used that $\trace A_4=0$.
Moreover,
$$
|A_4 E_2|^2=|A_4 E_1|^2\,,
$$
thus $|A_4|^2=2|A_4 E_1|^2=0$, that is $A_4=0$.
The second fundamental form $B$ of the immersed surfaces becomes
$$
B(E_1,E_1)=\lambda_1\, E_3\,,\quad B(E_2,E_2)=\lambda_2\, E_3\,,\quad B(E_1,E_2)=0\,,
$$
from which we obtain
$$
\begin{array}{ll}
\nabla_{E_1}E_1=\omega_1^2(E_1)\,E_2\,,&\nabla_{E_1}E_2=-\omega_1^2(E_1)\,E_1\,,\\
\nabla_{E_2}E_1=\omega_1^2(E_2)\,E_2\,,&\nabla_{E_2}E_2=-\omega_1^2(E_2)\,E_1\,.\\
\end{array}
$$
The Codazzi equation, when $X,Y,Z$ are tangent vector fields and $\eta$ is normal, takes the form
\begin{equation}\label{eq:codazzi-m2-n4}
\begin{aligned}
&X\langle B(Y,Z),\eta\rangle-\langle B(\nabla_X Y,Z),\eta\rangle - \langle B(Y,\nabla_X Z),\eta\rangle -\langle B(Y,Z),\nabla^{\perp}_X \eta\rangle\\& =
Y\langle B(X,Z),\eta\rangle-\langle B(\nabla_Y X,Z),\eta\rangle - \langle B(X,\nabla_Y Z),\eta\rangle -\langle B(X,Z),\nabla^{\perp}_Y \eta\rangle\,.
\end{aligned}
\end{equation}
We now use \eqref{eq:codazzi-m2-n4} four times:
\begin{itemize}
\item replacing $X=E_1$, $Y=E_2$, $Z=E_1$ and $\eta=E_3$ gives
\begin{equation}\label{eq:1213}
E_2\lambda_1=(\lambda_1-\lambda_2)\,\omega_1^2(E_1)\,;
\end{equation}
\item
replacing $X=E_1$, $Y=E_2$, $Z=E_2$ and $\eta=E_3$ we
 obtain
\begin{equation}\label{eq:1223}
E_1\lambda_2=(\lambda_1-\lambda_2)\,\omega_1^2(E_2)\,;
\end{equation}
\item replacing $X=E_1$, $Y=E_2$, $Z=E_1$ and $\eta=E_4$ yields
\begin{equation}\label{eq:1214}
\lambda_1\, \omega_3^4(E_2)=0\,;
\end{equation}
\item replacing $X=E_1$, $Y=E_2$, $Z=E_2$ and $\eta=E_4$  reduces to
\begin{equation}\label{eq:1224}
\lambda_2\, \omega_3^4(E_1)=0\,.
\end{equation}
\end{itemize}
Moreover, since $2\, H= (\lambda_1+\lambda_2)E_3$,
then
$$
|\lambda_1+\lambda_2|=2\, |H|=\cst\neq 0\,.
$$

Assume that $\omega_3^4(E_1)\neq 0$, then, from \eqref{eq:1224}, $\lambda_2=0$ on a open subset
and, taking into account \eqref{eq:1223}, $\omega_1^2(E_2)=0$. Next, since $\lambda_1=\cst$,
\eqref{eq:1213} and \eqref{eq:1214} give $\omega_1^2(E_1)=0$ and $\omega_3^4(E_2)=0$ respectively.
Thus the surface $M$ is flat.

Similarly, if $\omega_3^4(E_2)\neq 0$, we obtain that $\lambda_1=0$, the surface is flat and also $\omega_3^4(E_1)=0$.

Finally, replacing in the Gauss equation
$$
\langle R^N(X,Y)Z,W\rangle=\langle R(X,Y)Z,W\rangle+\langle B(X,Z),B(Y,W)\rangle-\langle B(X,W),B(Y,Z)\rangle
$$
$X=W=E_1$ and $Y=Z=E_2$, we get
\begin{equation}\label{eq:clambda12}
c=\lambda_1\,\lambda_2=0 \,\, ,
\end{equation}
which is a contradiction.
 \end{proof}

Because of Theorem~\ref{teo:m2n4biconservative}, we are left to analyze the case of CMC biconservative surfaces into $\R^4$. In this case, we have the following
explicit description.

\begin{proposition}
Let $M^2\imto\R^4$ be a proper (non PMC) biconservative surface with constant mean curvature different from zero. Then, locally, the surface is given by
\begin{equation}\label{eq:r-bi-m2-in-r4}
X(u,v)=(\gamma(u),v+a)=(\gamma^1(u),\gamma^2(u),\gamma^3(u),v+a),\quad a\in\R \,\, ,
\end{equation}
where $ \gamma:I\to \R^3$ is a curve into $\R^3$ parametrized by arc-length, with constant curvature $k\neq 0$ and torsion $\tau\neq0$.
Conversely, a surface into $\R^4$, parametrized by \eqref{eq:r-bi-m2-in-r4}, is a proper biconservative surface with constant mean curvature different from zero.
\end{proposition}
\begin{proof}
We use the local frame defined in \eqref{eq:frameei}. Then $2|H|=|\lambda_1+\lambda_2|$ and, using \eqref{eq:clambda12},  we can assume that $\lambda_1=\cst\neq 0$ and $\lambda_2=0$. Then
$$
\nabla^{\R^4}_{E_1}E_2=\nabla^{\R^4}_{E_2}E_2=0\,.
$$
Thus $E_2$ is the restriction to $M^2$ of a constant vector field of $\R^4$ which, up to an isometry of $\R^4$, we can choose as
$$
E_2=(0,0,0,1)\,.
$$
Let now $\sigma$ be an integral curve of $E_2$ parametrized by arc-length, that is
$$
\sigma(v)=(b_1,b_2,b_3,v+a)\,,\quad b_1,b_2,b_3,a\in\R\,.
$$
Let $p_0=(b_1,b_2,b_3,a)\in M$ and
$$
\gamma(u)=(\gamma^1(u),\gamma^2(u),\gamma^3(u),\gamma^4(u))\,,
$$
be an integral curve of $E_1$ with $\gamma(0)=p_0$. As $\langle\gamma'(u),E_2\rangle=0$, we get $\gamma^4(u)=a$.
Thus $\gamma$ lies in a hyperplane  of $\R^4$ orthogonal  to $E_2$. Then the surface can be locally parametrized by
\begin{equation}\label{eq:X-firstbis}
X(u,v)=(\gamma^1(u),\gamma^2(u),\gamma^3(u),v+a)\,.
\end{equation}
Using the calculation in the proof of Theorem~\ref{teo:m2n4biconservative}, we easily see that
$$
\begin{cases}
\nabla^{\R^4}_{E_1}E_1=\lambda_1\, E_3\\
\nabla^{\R^4}_{E_1}E_3=-\lambda_1\,E_1+\omega_3^4(E_1)\,E_4\\
\nabla^{\R^4}_{E_1}E_4=-\omega_3^4(E_1)\,E_3\,,
\end{cases}
$$
where, using Ricci's equation
$$
R^{\perp}(E_1,E_2)E_3=0\,,
$$
we obtain $E_2(\omega_3^4(E_1))=0$, that is $\omega_3^4(E_1)$ depends only on $u$. We can assume that $\lambda_1>0$.
 Thus $\{E_1, E_3, E_4\}$ is a Frenet's frame along $\gamma(u)$ and so $\gamma$ is a curve with constant
 curvature $k=\lambda_1$  and  torsion given by $\tau=\tau(u)=\omega_3^4(E_1)$. Moreover,  since $\nabla^{\perp}_{E_1} H=|H| \nabla^{\perp}_{E_1} E_3=(k\,\tau)/2\,E_4$ and the surface is not PMC, we conclude that $\tau\neq0$ at any point of $M$.

We now prove the converse. For this, suppose that the surface is parametrized by \eqref{eq:r-bi-m2-in-r4} and let $\{T=(\gamma'(u),0),N(u),B(u)\}$
be a Frenet frame along $\gamma$. Note that $\{N(u),B(u)\}$ is an orthonormal frame of the normal bundle of the surface. Then a straightforward computation gives:
$$
\nabla^{\R^4}_{\partial_u}\partial_u=X_{uu}= k\, N(u)\,,\quad \nabla^{\R^4}_{\partial_u}\partial_v=X_{uv}=0\,,\quad\nabla^{\R^4}_{\partial_v}\partial_v= X_{vv}=0\,,
$$
from which
$$
B(\partial_u,\partial_u)=k\, N(u)\,,\quad B(\partial_u,\partial_v)=0\,,\quad B(\partial_v,\partial_v)=0.
$$
Furthermore,
$$
\nabla_{\partial_u} ^{\perp} N(u)=\tau(u)\, B(u)\,,\quad \nabla_{\partial_v}^{\perp} N(u)=0\,.
$$
Thus $H=(k/2)\, N(u)$ and $\nabla_{\partial_u}^{\perp} H=(k\,\tau/2)\,B(u)$, showing that the surface is CMC but not PMC.
Finally, since the surface is CMC from  \eqref{eq:biconservative-general-n4}, to prove that it is  biconservative we only need to check that
$$
 A_{\nabla^\perp_{\partial_u}{N}}(\partial_u) +A_{\nabla^\perp_{\partial_v}{N}}(\partial_v) =
 \tau(u)\, A_{B}(\partial_u)=0\,.
 $$
 Indeed,
 $$
 A_{B}(\partial_u)=-\left(\nabla^{\R^4}_{\partial_u} B(u)\right)^{\top}=\left(\tau(u)\, N(u)\right)^{\top}=0\,.
 $$
\end{proof}

\end{document}